\documentclass[12pt,eqno,sumlimits]{amsart}

\usepackage{amssymb, amscd, amsmath, epsfig, mathtools}
\usepackage[utf8]{inputenc}
\usepackage{amsthm}
\usepackage{enumerate}
\usepackage{xcolor}
\usepackage{scalerel}
\usepackage{soul}
\usepackage{tikz-cd}
\usepackage{esint}
\usepackage{slashed}
\usepackage{hyperref}

\UseRawInputEncoding

\usepackage[margin=1.0in]{geometry}

\newtheorem{thm}{Theorem}[section]
\newtheorem*{theorem*}{Theorem}

\newtheorem{cor}{Corollary}[section]
\newtheorem{lem}{Lemma}[section]
\newtheorem{prop}{Proposition}[section]
\theoremstyle{definition}
\newtheorem{rem}{Remark}[section]

\newcommand{\C}{\mathbb C}
\newcommand{\R}{\mathbb R}

\newcommand{\Z}{\mathbb Z}

\newcommand{\calL}{\mathcal L}

\usepackage{mathtools}

\newcommand{\diff}{{\rm Diff}}
\newcommand{\ism}{{\rm Isom}}
\newcommand{\bm}{\overline M}

\newcommand{\bmp}{{\overline M}_p}

\newcommand{\kk}{2k-1}
\newcommand{\itm}{[0,1]\times M}
\newcommand{\itbm}{[0,1]\times \bm}
\newcommand{\ints}{\int_{S^1}}
\newcommand{\PP}{\mathbb P}

\newcommand{\sgn}{{\rm sgn}}

\begin{document}

\title[Geometry of Loop Spaces II: Corrections]{The geometry of loop spaces II: Corrections} 
\author[Y. Maeda]{Yoshiaki Maeda}
\address{Tohoku Forum for Creativity, Tohoku University}
\email{yoshimaeda@m.tohoku.ac.jp}
\author[S. Rosenberg]{Steven Rosenberg}
\address{Department of Mathematics and Statistics, Boston University}
\email{sr@math.bu.edu}
\author[F. Torres-Ardila]{Fabi\'an Torres-Ardila}
\address{Mauricio Gast\'on Institute for Latino Community Development and Public Policy,
University of Massachusetts Boston}
\email{Fabian.Torres-Ardila@umb.edu}

\maketitle

\section{Introduction}

This article contains corrections and extensions of results in \cite{MRT4}.  The main correction is that theorems stating $|\pi_1(\diff(\bmp))| = \infty$ must be replaced with $|\pi_1(\ism(\bmp))| = \infty$, where $\diff(\bmp)$, $\ism(\bmp)$ refer to the identity component of the diffeomorphism groups and isometry groups of certain $5$-manifolds $\bmp.$ In particular, this applies to Cor.~3.9, Thm.~3.10, Cor.~3.11, Ex.~3.12, Ex.~3.13, Ex.~3.15, 
Prop.~ 3.16, Ex.~3.17, Prop.~ 3.18, Prop.~3.20 in \cite{MRT4}.  As a new result, we find an infinite family of pairwise non-isometric metrics $h_p, p\in \Z^+,$ on $S^5$ such that 
$|\pi_1(\ism(S^5, h_p))| = \infty.$

In the proof of \cite[Prop.~3.4]{MRT4}, we implicitly used that  the Wodzicki-Chern-Simons $CS^W_{2k-1}\in \Lambda^{2k-1}(L\bm)$ is closed for a  $(2k-1)$-dimensional closed manifold $\bm$, in the notation of \cite{MRT4}.
This is not correct in general. In fact, we can only prove that $F^{L,*}CS^W_{2k-1}$ is closed as a $(2k-1)$-form on $[0,1]\times \bm$.  Here $F:[0,1]\times S^1\times \bm\to \bm$ and we assume
$F^L:[0,1]\times \bm\to L\bm$, $ F^L(x^0,m)(\theta) = F(x^0,\theta, m)$, is an isometry of $\bm$ for all $(x^0,\theta).$  This is proved in Prop.~\ref{prop:2.1}.  In Prop.~\ref{prop:two}, we then recover \cite[Prop.~3.4]{MRT4} for isometries.  
In Cor.~\ref{cor:one}, we restate the main results in \cite{MRT4} for isometry groups.  In Prop.~\ref{prop:new}, we construct the metrics $h_p$ as above.

In \S3, we give other corrections to \cite{MRT4}.  Most are minor, but in Lem.~\ref{lem:3.1} we note that   $C^k$-mapping spaces are homotopy equivalent for different $k$, while the homotopy equivalence for different Sobolev topologies  is unclear.  Thus we should have used $C^k$ topologies in \cite{MRT4}.

Even though this paper now concentrates on isometry groups, we should also add references to recent work on diffeomorphism groups omitted in the Introduction in \cite{MRT4}: 
    Bamler-Kleiner \cite{BK}, Baraglia \cite{Bar}, 
   Iida-Konno-Mukherjee-Taniguchi \cite{IKMT}, 
   Kronheimer-Mrowka \cite{KrM}, Kupers-Randal-Williams \cite{KRW}, Lin-Mukherjee \cite{LM}, Konno-Tanaguchi \cite{KT}, and Watanabe \cite{W}. This list is not exhaustive.

\section{Replacing diffeomorphism groups with isometry groups}
In \S2.1, we state the main results and prove  Prop.~\ref{prop:new}.  In \S2.2, we prove the usual naturality property $d_Mf^*\omega = f^* d_N\omega$ for maps $f:M\to N$ of Banach manifolds and $\omega \in \Lambda^*(N).$
In \S2.3, we recall the local coordinate expression for $CS^W.$ In \S2.4, we obtain the local coordinate expression for $d_{\itbm} F^{L^*} CS^W = F^{L^*} d_{L\bm} CS^W$ (Prop.~\ref{prop:2.4}). In \S2.5, we simplify this expression if $F^L(x^0,\theta)\in \diff(\bm)$ for all $(x^0, \theta).$  In \S2.6, we finish the proof of
Prop.~\ref{prop:2.1} provided $F^L(x^0,\theta)\in \ism(\bm)$.

\subsection{The main results}

Let $\bm$ be a closed $(2k-1)$-manifold. For  smooth maps $a:S^1\times \bm\to \bm$,
$F:[0,1]\times S^1\times \bm\to \bm$,  set
$a^L:\bm\to L\bm$ by $a^L(m)(\theta) = a(\theta,m),$ and 
$F^L:[0,1]\times \bm\to L\bm$ by $ F^L(t,m)(\theta) = F(t,\theta, m).$  
The key result Prop.~3.4 in \cite{MRT4} is replaced with two results.

\begin{prop} \label{prop:2.1} Let $a_0, a_1:S^1\times \bm\to \bm$ be $S^1$ actions on a closed Riemannian $(2k-1)$-manifold $(\bm,g)$ with $a_\theta := a(\theta,\cdot) \in \ism(\bm,g)$ for all $\theta\in S^1.$  Let $F:[0,1]\times S^1\times \bm\to \bm$ be a smooth homotopy from $a_0$ to $a_1$ ({\it i.e.,} $F(0,\cdot,\cdot) = a_0, F(1, \cdot,\cdot) = a_1$) with 
$F(x^0,\theta,\cdot)\in \ism(\bm,g)$ for all $(x^0,\theta).$  Then 
$$d_{\itm}F^{L,*}CS^W_{\kk} =0.$$
\end{prop}

Here $CS^W_{\kk} $ depends on $g$  (\ref{CSW}). The proof is given in \S2.2 (see Lem.~\ref{lem:2.3}).  The proof fails if $\ism(\bm,g)$ is replaced with $\diff(\bm).$

From now on, we denote $\ism(\bm,g)$ by $\ism(\bm)$, except when we use the explicit metric $\bar g$ on $\bmp$ in \cite{MRT4}.

\begin{prop}[Prop.~3.4 in \cite{MRT4}] \label{prop:two} Let dim$(\bm)=2k-1.$ Let $a_0, a_1:S^1\times \bm\to \bm$ be actions through isometries as above.

(i)  If $a_0$ and $a_1$ are homotopic through isometries, then 
$\int_{\bm} a_0^{L,*}CS^W_{\kk} =  \int_{M} a_1^{L,*}CS^W_{\kk}$. 

(ii)  If there exists an action $a$ through isometries with $\int_M a^{L,*} CS^W_{\kk} \neq 0,$ then
  $\pi_1(\ism(M))$  is infinite.

\end{prop}

\begin{proof} 
(i)  This is just Stokes' Theorem:  for $i_{x^0}:\bm\to [0,1]\times \bm, i_{x^0}(\bm) = (x^0,m),$ and $F$ the homotopy in Prop.~\ref{prop:2.1}, 
\begin{align*}\int_{\bm} a_0^{L,*}CS^W_{\kk} -  \int_{\bm} a_1^{L,*}CS^W_{\kk} &= 
\int_{\bm} i_0^*F^{L,*}CS^W_{\kk} -  \int_{\bm} i_1^* F^{L,*}CS^W_{\kk}\\
&= \int_{\itm} d_{\itm} F^*CS^W_{\kk} =0,
\end{align*}
by Prop.~\ref{prop:2.1}.

(ii) Let $a_n$ be the $n^{\rm th}$ iterate of 
 $a$, i.e. $a_n(\theta,m) =
a(n\theta,m).$  
We claim that 
 $\int_{\bm}a_n^{L,*}CS^W_{\kk} =
n\int_{\bm} a^{L,*}CS^W_{\kk}$.  By \cite[(2.5)]{MRT4}, every term in $CS^W_{\kk}$ is of the
form $\ints\dot\gamma(\theta) f(\theta)$, where $f$ is a periodic function on the
circle.  Each loop $\gamma\in
a^L_1(\bm)$ corresponds to the loop $\gamma(n\cdot)\in a^L_n(\bm).$  Therefore the term
$\ints\dot\gamma(\theta) f(\theta)$ is replaced by 
$$\ints \frac{d}{d\theta}\gamma(n\theta) f(n\theta)d\theta 
 = n\int_0^{2\pi} \dot\gamma(\theta)f(\theta)d\theta.$$
Thus $\int_{\bm}a_n^{L,*}CS^W_{\kk} = n\int_{\bm}a^{L,*}CS^W_{\kk}.$
 By (i), 
$a_n$ and $a_m$ are not homotopic through isometries.  By a straightforward modification of \cite[Lem.~3.3]{MRT4}, 
 the $[a^L_n]\in 
\pi_1(\ism(\bm))$
are all distinct.  
\end{proof}

This  recovers the results stated in the introduction, as the calculations of $a^{L,*}CS^W_{\kk}$ are unchanged.

\begin{cor} \label{cor:one} The results Cor.~3.9, Thm.~3.10, Cor.~3.11, Ex.~3.12, Ex.~3.13, Ex.~3.15, 
Prop.~ 3.16, Ex.~3.17, Prop.~ 3.18, Prop.~3.20 in \cite{MRT4} are correct with $\diff(\bmp)$ replaced by
$\ism(\bmp).$

In particular (Thm.~3.10), let $(M^4, J, g, \omega)$ be a compact integral K\"ahler surface, and let $\bmp$ be the circle bundle associated to $p[\omega]\in H^2(M, \Z)$ for $p\in \Z.$  Let $\bar g = \bar g_p$ be the metric on $\bmp$ given in \S3.2 of \cite{MRT4}. Then
the loop of diffeomorphisms of $\bmp$ given by rotation in the circle fiber is an element of infinite order in $\pi_1(\ism(\bmp, \bar g))$  if $|p| \gg 0.$

As another example (Ex.~3.13, Prop.~3.14), for the lens space $\overline{\C\PP^2}_p={\mathcal L}_p = S^5/\Z_p$, $p\neq 1,$ with the metric $\bar g$ on \cite[p.~500]{MRT4}, $\pi_1(\ism(\mathcal L_p, \bar g))$ is infinite.
\end{cor}

\begin{rem}  The  metric on $\mathcal L_p$ coming from the standard metric on $S^5$ has isometry group
$SO(6)/\Z_p$, which has finite fundamental group.  Thus $\bar g$ is not the standard metric.  This can also be seen from \cite[Prop.~3.6]{MRT4}, since the sectional curvatures of $\calL_p$ with the standard metric are independent of $p$.
\end{rem} 

This implies a result about isometries of $S^5$ that was missed in \cite{MRT4}.

\begin{prop} \label{prop:new}
$S^5$ admits an infinite family of pairwise nonisometric metrics $h_p$ with \\
$|\pi_1(\ism(S^5, h_p))| = \infty.$    
\end{prop}
\begin{proof}
The metric $\bar g_p, p \neq 1,$ on $\mathcal L_p$ lifts to a locally isometric metric $h_p$ on $S^5.$  As above, the $h_p$ have different sectional curvatures  
\cite[Prop.~3.6]{MRT4}, so they are not isometric.  The circle action $a:S^1 \times \calL_p\to \calL_p $ does not lift to a circle action on $S^5$, 
but $a_p$ does lift to an action $b_p:S^1\times S^5\to S^5$.  Since $\bar g_p, h_p$ are locally isometric, we get $CS^W_{\kk,S^5}(\gamma)(\theta) = CS^W_{\kk,\calL_p}(\pi\circ \gamma)(\theta), $ where $\pi:S^5\to\calL_p$ is the projection.  This implies
$$\int_{S^5} b^{L,*} CS^W_{5,S^5} =   \int_{\calL_p} a_p^{L,*} CS^W_{5,\calL_p} = p
\int_{\calL_p} a^{L,*} CS^W_{5,\calL_p} \neq 0.$$
As above, this implies $|\pi_1(\ism(S^5, h_p))| = \infty.$
\end{proof}

Note that $\pi_1(\ism(S^5, g_0)) = \Z_2$ for the standard metric $g_0$, and for a generic metric $g$ on $S^5$, we expect
$\ism(S^5, g) = \{{\rm Id}\},$ $\pi_1(\ism(S^5, g))= 1.$  Thus the metrics $h_p$ are ``special, but not too special."

\subsection{Naturality of pullbacks of forms on Banach manifolds}

We first derive the surely known result that the proof that $d_{M}f^*\omega = f^*d_N \omega$ (for $f:M\to N$ a differentiable map between finite dimensional  manifolds and $\omega\in \Omega^s(N)$) extends to infinite dimensional smooth Banach manifolds like $LM.$  On an infinite dimensional smooth manifold $N$, the exterior derivative can only be defined by the Cartan formula:
\begin{align*}d_N\omega(x^0,\ldots X_{s})_p &= \sum_i (-1)^i X_i(\omega(x^0,\ldots,\widehat{X_i},\ldots, X_{s})\\
&\qquad 
+ \sum_{i<j} (-1)^{i+j}\omega([X_i,X_j], x^0,\ldots, \widehat{X_i},\ldots, \widehat{X_j},\ldots,X_{s}),
\end{align*}
where $X_i\in T_pN$ are extended to vector fields near $p$ using a chart map (see {\it e.g.,} \cite[\S33.12]{KM}).  
\begin{lem} \label{lem:2.1}  Let $f:M\to N$ be a smooth map between smooth Banach manifolds, and let $\omega\in \Omega^*(N).$
Then $d_Mf^*\omega = f^*d_N \omega$.  In particular, $d_{\itbm}F^{L,*}CS^W_{\kk} = F^{L^,*} d_{L\bm} CS^W_{\kk}$ in Prop.~\ref{prop:2.1}.
\end{lem} 

  \begin{proof} First assume that $f$ is an immersion on a neighborhood $U_p$ of a fixed $p\in M$. 
For fixed vector fields $Y_i$ on $U_p$, set
 $g:f(U_p)\to \R$, $g(n) = \omega(f_*Y_1,\ldots, f_*Y_s)_{n}$. 
We have $(g\circ f)(m) =  \omega(f_*Y_1,\ldots, f_*Y_s)_{f(m)}$.  Thus the identity $X_m(g\circ f) = (f_*X)
_{f(m)}(g)$
becomes 
$$X_m(\omega(f_*Y_1,\ldots, f_*Y_s)) = (f_*X)_{f(m)}(\omega(f_*Y_1,\ldots, f_*Y_s)).$$
Dropping $m, f(m)$, we get
\begin{align*}f^*d_N\omega(x^0,\ldots, X_{s}) &=d_N\omega(f_*x^0,\ldots, f_*X_{s})\\
&=\sum_i (-1)^i f_*X_i(\omega(f_*x^0,\ldots,\widehat{f_*X_i},\ldots, f_*X_{s})\\
&\qquad 
+ \sum_{i<j} (-1)^{i+j}\omega([f_*X_i,f_*X_j], f_*x^0,\ldots, \widehat{f_*X_i},\ldots, \widehat{f_*X_j},\ldots,f_*X_{s})\\
&= \sum_i (-1)^i X_i(\omega(f_*x^0,\ldots,\widehat{f_*X_i},\ldots, f_*X_{s})\\
&\qquad 
+ \sum_{i<j} (-1)^{i+j}\omega(f_*[X_i,X_j], f_*x^0,\ldots, \widehat{f_*X_i},\ldots, \widehat{f_*X_j},\ldots,f_*X_{s})\\
&= d_M \omega (f_*x^0,\ldots, f_*X_s) = d_Mf^*\omega(x^0,\ldots,X_s),
\end{align*}
where we use $[f_*X_i,f_*X_j] = f_*[X_i,X_j]$ for immersions.

In general, 
consider the graph $G:M\to M\times N$, $G(m) = (m, f(m)).$  Then $\pi_N\circ G = f$ for the projection $\pi_N:M\times N\to N$. (We similarly define $\pi_M.$)
$G$ is an immersion, with $G_*(Y) = (Y, f_*Y)$ taking a vector field on $M$ to a well-defined vector field on $M\times N.$   

Fix $(m_0, n_0)\in M\times N$, and set $i_M:M\to N\times N$, $i_M:N\to M\times N$ by
$i_M(m) = (m, n_0), i_N(n) = (m_0,n).$   If  a vector $(X^0,Y_0)\in T_{(m_0,n_0)}M\times N$
is extended to a nearby vector field $(X,Y)$ with $X$ constant in $N$ directions and $Y$ constant in $M$ directions, 
it is straightforward to apply the Cartan formula to derive  the standard equality (usually abbreviated 
$d_{M\times N} = d_M + d_N$)
$$d_{M\times N}\alpha_{(m_0,n_0)} = \pi_M^*[d_M (i_M^*\alpha)_{m_0}] + \pi_N^*[d_N (i_N^*\alpha)_{n_0}],$$ 
for $\alpha\in \Omega^*(M\times N).$  
Since $\pi i_M:m\mapsto n_0$ (so $d_Mi_M^*\pi^*\omega = 0$) and $\pi i_N = {\rm id}$,  the argument above for the immersion $G$ yields
\begin{align*} d_Mf^*\omega &= d_M G^*\pi_N^*\omega = G^* d_{M\times N}\pi_N^*\omega 
= G^*[\pi_M^*d_M i_M^*\pi^*\omega + \pi_N^*d_N i_N^*\pi^*\omega]\nonumber\\
&= G^*\pi_N^*d_N i_N^*\pi^*\omega = f^* d_N \omega.
\end{align*}
\end{proof}

\subsection{Local coordinates expression}

We work in local coordinates $(x^0, x) = (x^0, x^1,\ldots x^{2k-1})$ on $\itbm$. 
Let 
\begin{equation} 
\label{K_tensor}
\begin{aligned}
& K_{\nu \lambda_1 \cdots \lambda_{2k-1}} \\
&\quad \quad 
= \sum_{\sigma } \sgn({\sigma})
R_{\lambda_{\sigma (1) e_1 \nu}}{}^{e_{2}} 
R_{\lambda_{\sigma (2) } \lambda_{\sigma (3)} e_{3}}{}^{e_1} 
R_{\lambda_{\sigma (4) } \lambda_{\sigma (5)} e_{1}}{}^{e_{3}} 
\cdots 
R_{\lambda_{\sigma (2k-2) } \lambda_{\sigma (2k-1)} e_{2}}{}^{e_{k-1}},
\end{aligned}
\end{equation} 
for $\sigma$ a permutation of $\{1,\ldots,2k-1\}$, and where $R_{ijk}^{\ \ \ \ell}$ are the components of the curvature tensor of the metric on $M$.
\begin{equation}\label{K}K_{\nu \lambda_1 \cdots \lambda_{2k-1}}dx^\nu\otimes dx^{\lambda_1}\wedge\ldots\wedge dx^{\lambda_{2k-1}}
\end{equation}
is the local expression of an element of 
$\Omega^1( {\bm}) \otimes \Omega^{2k-2} (\bm).$
For $\gamma\in L\bm$ and $X_{\gamma,i}\in T_\gamma L\bm$, we set  
\begin{equation}
\label{CSW}  
CS^W (\gamma ) (X_{\gamma, 1} , \cdots , X_{\gamma , 2k-1}) \\
= \int_0^{2\pi}
 K_{\nu \lambda_1 \cdots \lambda_{2k-1}} (\gamma (\theta))
{\dot \gamma}^{\nu} (\theta )
X_{\gamma ,1}^{\lambda_{1}}(\theta ) \cdots X_{\gamma , 2k-1}^{\lambda_{2k-1}}
 d\theta.
\end{equation}
Then $CS^W 
\in \Omega^{2k-1} (L{\bm})$, since we have contracted out the $\nu$ index.  Since the integrand in (\ref{CSW}) is tensorial, we can integrate over $[0, 2\pi]$ even if the image of $\gamma$ does not lie in one coordinate chart.

\subsection{Computing $F^{L,*} d_{L\bm} CS^W$}
We have
\begin{align*}
 \MoveEqLeft{(d_{L\bm} CS^W_\gamma ) 
 (X_{\gamma ,0}, X_{\gamma ,1}, \cdots , X_{\gamma ,2k-1} ) }\nonumber\\
 &=
 \sum_{a=0}^{2k-1} 
 (-1)^a X_{\gamma ,a} 
  (CS^W  (X_{\gamma ,0}, \cdots , 
  \widehat{X_{\gamma ,a}},
   \cdots , X_{\gamma ,2k-1})) \\
&\qquad +
\sum_{a<b} (-1)^{a+b} 
 (CS^W  ( [X_{\gamma , a}, X_{\gamma , b}], 
 X_{\gamma ,0}, \cdots , \widehat{X_{\gamma ,a}},
   \cdots ,  \widehat{X_{\gamma , b}}, X_{\gamma ,2k-1})) \nonumber  \\
   &:= \sum_a (1)_a + \sum_{a<b}(2)_{a,b}\nonumber.
 \end{align*}
Let 
$\gamma_s (\theta) \in L{\bm} $ be a family of loops with 
$\gamma_0 (\theta ) = \gamma (\theta)$, 
${\frac{d}{ds}}\bigl|_{s=0} \gamma_s = X_{\gamma , a} $. 
Then
\begin{align}\label{eq:4}
\MoveEqLeft{  X_{\gamma ,a} 
  (CS^W  (X_{\gamma ,0}, \cdots , \widehat{X_{\gamma ,a}},
   \cdots , X_{\gamma ,2k-1})) }\nonumber \\
 &= 
\int_{0}^{2\pi} 
 {\frac{d}{ds}}\biggl|_{s=0} 
\left[K_{\nu \lambda_0 \cdots \widehat{\lambda_a}\cdots \lambda_{2k-1}} 
(\gamma_s (\theta )) {\dot{\gamma_s}^\nu} 
X_{\gamma_s,0}^{\lambda_0}
\cdots \widehat{X_{\gamma_s ,a}^{\lambda_a}}\cdots
X_{\gamma_s,2k-1}^{\lambda_{2k-1}}
 d\theta \right] \nonumber\\
&=
\int_{0}^{2\pi} 
 \partial_{x^\mu}  K_{\nu \lambda_0 \cdots \widehat{\lambda_a}\cdots\lambda_{2k-1}}  
(\gamma(\theta)) 
X_{\gamma, a}^{\mu} \dot{\gamma}^\nu (\theta) 
X_{\gamma ,0}^{\lambda_0 } (\theta) 
\cdots \widehat{X_{\gamma ,a}^{\lambda_a}}(\theta)\cdots
X_{\gamma ,2k-1}^{\lambda_{2k-1} } (\theta) 
 d\theta  \nonumber\\
&\qquad +
\int_{0}^{2\pi} 
K_{\nu \lambda_0 \cdots \widehat{\lambda_a}\cdots \lambda_{2k-1}} (\gamma (\theta) )
\dot{X}_{\gamma ,a}^\nu (\theta ) 
X_{\gamma ,0}^{\lambda_0 } (\theta) 
\cdots\widehat{X_{\gamma ,a}^{\lambda_a}}(\theta)\cdots
X_{\gamma ,2k-1}^{\lambda_{2k-1} } (\theta) \\
&\qquad +
\int_{0}^{2\pi} 
K_{\nu \lambda_0 \cdots\widehat{\lambda_a}\cdots \lambda_{2k-1}} (\gamma (\theta) )
\dot{\gamma}^\nu (\theta) 
\left(\delta_{X_{\gamma, a}}X_{\gamma, 0}^{\lambda_0}\right)
X_{\gamma ,1}^{\lambda_2 } (\theta) 
\cdots\widehat{X_{\gamma ,a}^{\lambda_a}}(\theta)\cdots
X_{\gamma ,2k-1}^{\lambda_{2k-1} } (\theta) \nonumber\\
&\qquad + \cdots\nonumber \\
&\qquad +
\int_{0}^{2\pi} 
K_{\nu \lambda_0 \cdots\widehat{\lambda_a}\cdots \lambda_{2k-1}} (\gamma (\theta) )
\dot{\gamma}^\nu (\theta) 
X_{\gamma ,0}^{\lambda_0 } (\theta) 
\cdots\widehat{X_{\gamma ,a}^{\lambda_a}}(\theta)\cdots
X_{\gamma ,2k-2}^{\lambda_{2k-2} } (\theta) 
\left(\delta_{X_{\gamma, a}}X_{\gamma, 2k-1}^{\lambda_{2k-1}}\right).\nonumber
\end{align}
Denote the last three lines of (\ref{eq:4}) by $(\ref{eq:4})_a.$ Then it is easily seen that
\begin{equation*}
 \sum_{a=0}^{2k-1} (-1)^a  (\ref{eq:4})_a
   + \sum_{a<b} (2)_{a,b} =0
\end{equation*}
Therefore, 
\begin{align*}
 \MoveEqLeft{(d_{L\bm} CS^W_\gamma ) 
 (X_{\gamma ,0}, X_{\gamma ,1}, \cdots , X_{\gamma ,2k-1} ) } \\
&= \sum_{a=0}^{2k-1} (-1)^a
\int_{0}^{2\pi} 
 \partial_{x^\mu}  K_{\nu \lambda_0 \cdots \widehat{\lambda_a}\cdots\lambda_{2k-1}}  
(\gamma(\theta)) 
X_{\gamma, a}^{\mu} \dot{\gamma}^\nu (\theta) 
X_{\gamma ,0}^{\lambda_0 } (\theta) 
\cdots \widehat{X_{\gamma ,a}^{\lambda_a}}(\theta)\cdots
X_{\gamma ,2k-1}^{\lambda_{2k-1} } (\theta) 
 d\theta  \nonumber\\
&\qquad +\sum_{a=0}^{2k-1} (-1)^a
\int_{0}^{2\pi} 
K_{\nu \lambda_0 \cdots \widehat{\lambda_a}\cdots \lambda_{2k-1}} (\gamma (\theta) )
\dot{X}_{\gamma ,a}^\nu (\theta ) 
X_{\gamma ,0}^{\lambda_0 } (\theta) 
\cdots\widehat{X_{\gamma ,a}^{\lambda_a}}(\theta)\cdots
X_{\gamma ,2k-1}^{\lambda_{2k-1} } (\theta).
\end{align*}

For the pullback, we consider 
$( F^{L,*} d_{L{\bm}}  CS^W) 
(\partial_{x^0}, \partial_{x^1}, 
\cdots, \partial_{x^{2k-1}})$
as a function on $[0,1]\times U$, where $(U,x=(x^1,\ldots, x^{2k-1}))$ is a coordinate chart on $\bm.$
Then
\begin{align}\label{eq:pullback}
\MoveEqLeft{( F^{L,*} d_{L{\bm}}  CS^W) 
(\partial_{x^0}, \partial_{x^1}, 
\cdots, \partial_{x^{2k-1}})_{(x^0,x)} } \nonumber\\
&=
d_{L{\bm}} CS^W
 \left( F^L_* \partial_{x^0},  F^L_* \partial_{x^1}, 
\cdots,   F^L_* \partial_{x^{2k-1}} \right)_{F(x^0,x)} \nonumber\\
&=
d_{L{\bm} } CS^W \left(
\frac{\partial F^{\lambda_0}}{\partial x^0} \partial_{x^{\lambda_0}}, 
\frac{\partial F^{\lambda_1}}{\partial x^1}\partial_{x^{\lambda_1}},
\cdots ,
\frac{\partial F^{\lambda_{2k-1}}}{\partial x^{2k-1}}\partial_{x^{\lambda_{2k-1}}}\right)_{F(x^0,x)}
\\
&= 
\sum_{a=0}^{2k-1} (-1)^a 
\int_0^{2\pi} \partial_{x^\mu} 
K_{\nu \lambda_0 \cdots\widehat{\lambda_a}\cdots \lambda_{2k-1}} 
(F(x^0,\theta, x))
\frac{\partial F^\mu}{\partial x^a}
\frac{\partial F^\nu}{\partial \theta}
\frac{\partial F^{\lambda_0}}{\partial x^0}\cdots
\widehat{\frac{\partial F^{\lambda_a}}{\partial x^a} }
\cdots
\frac{\partial F^{\lambda_{2k-1}}}{\partial x^{2k-1}}
 d\theta     \nonumber\\
&\qquad +
\sum_{a=0}^{2k-1} (-1)^a  
\int_0^{2\pi}
K_{\nu \lambda_0 \cdots \widehat{\lambda_a} \cdots \lambda_{2k-1}} 
(F(x^0, \theta , x))
\frac{\partial^2 F^\nu}{\partial x^a \partial \theta}
\frac{\partial F^{\lambda_0}}{\partial x^0}
\cdots
\widehat{\frac{\partial F^{\lambda_a}}{\partial x^a}}
\cdots
\frac{\partial F^{\lambda_{2k-1}}}{\partial x^{2k-1}}
 d\theta\nonumber
\end{align}
One term in the last equation in (\ref{eq:pullback}) vanishes.  The proof is in the Appendix.

\begin{lem} \label{lem:2.2a}
\begin{align*}
\int_0^{2\pi} \sum_{a=0}^{2k-1} (-1)^a 
\partial_{x^\mu} 
K_{\nu \lambda_0 \cdots \widehat{\lambda_a}\cdots\lambda_{2k-1}} 
(F(x^0,\theta, x))
\frac{\partial F^\nu}{\partial \theta}
\frac{\partial F^\mu}{\partial x^a}
\frac{\partial F^{\lambda_0}}{\partial x^0}\cdots
\widehat{\frac{\partial F^{\lambda_a}}{\partial x^a} }
\cdots
\frac{\partial F^{\lambda_{2k-1}}}{\partial x^{2k-1}}
 d\theta 
&=0.
\end{align*}
\end{lem}

Thus, we have
\begin{prop}\label{prop:2.4}
\begin{align}\label{eq:four}
\MoveEqLeft{F^{L,*} d_{L{\bar M}} CS^W 
(\partial_{x^0}, \partial_{x^1}, 
\cdots, \partial_{x^{2k-1}})_{(x^0,x)}}\nonumber\\
&=
\sum_{a=0}^{2k-1} (-1)^a  
\int_0^{2\pi}
K_{\nu \lambda_0 \cdots \widehat{\lambda_a} \cdots \lambda_{2k-1}} 
(F(x^0, \theta , x))
\frac{\partial^2 F^\nu}{\partial x^a \partial \theta}
\frac{\partial F^{\lambda_0}}{\partial x^0}
\cdots
\widehat{\frac{\partial F^{\lambda_a}}{\partial x^a}}
\cdots
\frac{\partial F^{\lambda_{2k-1}}}{\partial x^{2k-1}}
 d\theta.
\end{align}
\end{prop}

\subsection{Homotopies of loops of diffeomorphisms}

We now make the assumption that
\begin{equation}\label{eq:diff} F(x^0, \theta,\cdot):\bm\to\bm \ \text{is a diffeomorphism for all} \ (x^0,\theta)\in [0,1]\times S^1.
\end{equation}
Then $\{F_*(\partial/\partial x^i)\}_{i=1}^{2k-1}$ is a basis of $T_{F(x^0,\theta,x)}\bm$ for all $(x^0,\theta,x).$
Therefore, there exist functions $\alpha^i =\alpha^i(x^0,\theta,x)$, $i = 1,\ldots,2k-1$, such that
\begin{equation}\label{eq:five} F_*\left(\frac{\partial}{\partial x^0}\right) = 
\alpha^iF_*\left(\frac{\partial}{\partial x^i}\right).
\end{equation}
Using coordinates $y^i = y^i(x^0,\theta, x)$ near $y = F(x^0,\theta,x)$, we have
$$F_*\left(\frac{\partial}{\partial x^0}\biggl|_{_{(x^0, \theta,x)}}\right) = \frac{\partial F^\lambda}{\partial x^0} \frac{\partial}{\partial y^\lambda}\biggl|_{_{y}} \in T_y\bm,\ 
F_*\left(\frac{\partial}{\partial x^i}\biggl|_{_{(x^0, \theta,x)}}\right) = \frac{\partial F^\lambda}{\partial x^i} \frac{\partial}{\partial y^\lambda}\biggl|_{_{y}} \in T_y\bm.$$
Thus
\begin{equation}\label{eq:five}
\frac{\partial F^\lambda}{\partial x^0} = \alpha^i\frac{\partial F^\lambda}{\partial x^i},\ 
\frac{\partial^2 F^\lambda}{\partial\theta \partial x^0} = 
\frac{\partial \alpha^i}{\partial\theta} \frac{\partial F^\lambda}{\partial x^i}
+ \alpha^i\frac{\partial^2 F^\lambda}{\partial\theta \partial x^i}.
\end{equation}
Plugging (\ref{eq:five}) into (\ref{eq:four}) gives
\begin{align}\label{eq.8}
\MoveEqLeft{F^{L,*} d_{L{\bar M}} CS^W 
(\partial_{x^0}, \partial_{x^1}, 
\cdots, \partial_{x^{2k-1}})}\nonumber\\
&= 
 \int_0^{2\pi} K_{\nu \lambda_1  \cdots \lambda_{2k-1}}
\left( \frac{\partial \alpha^i}{\partial\theta} \frac{\partial F^\nu}{\partial x^i}
+ \alpha^i\frac{\partial^2 F^\nu}{\partial\theta \partial x^i}\right)
\frac{\partial F^{\lambda_1}}{\partial x^1}
\cdots
\frac{\partial F^{\lambda_{2k-1}}}{\partial x^{2k-1}}
 d\theta\\
&\qquad  +\sum_{a=1}^{2k-1} 
(-1)^a  
\int_0^{2\pi}
K_{\nu \lambda_0 \cdots \widehat{\lambda_a} \cdots \lambda_{2k-1}} 
\frac{\partial^2 F^\nu}{\partial x^a \partial \theta} \left( \alpha^i\frac{\partial F^{\lambda_0}}{\partial x^i}
\right) \frac{\partial F^{\lambda_1}}{\partial x^1}
\cdots 
\widehat{\frac{\partial F^{\lambda_a}}{\partial x^a}}
\cdots
\frac{\partial F^{\lambda_{2k-1}}}{\partial x^{2k-1}}
 d\theta. \nonumber
\end{align}
The sum of the terms with the second partial derivatives vanishes:
\begin{lem}\label{lem:2.3a}
\begin{align}\label{eq:9}\MoveEqLeft{ 0= \int_0^{2\pi} K_{\nu \lambda_1\ldots\lambda_{2k-1}}
 \alpha^i\frac{\partial^2 F^\nu}{\partial\theta \partial x^i}
\frac{\partial F^{\lambda_1}}{\partial x^1}
\cdots
\frac{\partial F^{\lambda_{2k-1}}}{\partial x^{2k-1}}
 d\theta}\nonumber \\
&+ \sum_{a=1}^{2k-1} 
(-1)^a  
\int_0^{2\pi}
K_{\nu \lambda_0 \cdots \widehat{\lambda_a} \cdots \lambda_{2k-1}} 
\frac{\partial^2 F^\nu}{\partial x^a \partial \theta} \left( \alpha^i\frac{\partial F^{\lambda_0}}{\partial x^i}
\right) \frac{\partial F^{\lambda_1}}{\partial x^1}
\cdots 
\widehat{\frac{\partial F^{\lambda_a}}{\partial x^a}}
\cdots
\frac{\partial F^{\lambda_{2k-1}}}{\partial x^{2k-1}}
 d\theta.
 \end{align}
 \end{lem}
 
 The proof is in the Appendix.
 Changing the index $\nu$ to $\lambda_0$,  we have proved the following:
 \begin{lem}\label{lem:2.2}
 Under assumption (\ref{eq:diff}), we have
 \begin{equation}\label{eq:10}
 F^{L,*} d_{L{\bar M}} CS^W (\partial_{x^0}, \partial_{x^1}, \cdots, \partial_{x^{2k-1}}) = 
  \int_0^{2\pi} K_{\lambda_0 \lambda_1\ldots\lambda_{2k-1}}
\frac{\partial \alpha^i}{\partial\theta} \frac{\partial F^{\lambda_0}}{\partial x^i}
\frac{\partial F^{\lambda_1}}{\partial x^1}
\cdots
\frac{\partial F^{\lambda_{2k-1}}}{\partial x^{2k-1}}
 d\theta.
 \end{equation}
  \end{lem}

  \subsection{Homotopies by loops of isometries}
We now make the further assumption that
\begin{equation}\label{eq:isom} F^I(x^0,\theta) := F(x^0, \theta,\cdot):\bm\to\bm \ \text{is an isometry for all} \ (x^0,\theta)\in [0,1]\times S^1.
\end{equation}
Thus for fixed $(x^0,\theta)$, 
$$ g_{ij}(x) = (F^{L,*}g)_{ij}(x) =g_{\lambda\mu}(F(x^0,\theta,x)) \frac{\partial F^{L, \lambda}}{\partial x^i}
\biggl|_{_{(x^0,\theta,x)}}
\frac{\partial F^{L, \lambda}}{\partial x^i}\biggl|_{_{(x^0,\theta,x)}}.$$
With some notation dropped, it follows that
\begin{align}\label{eq:11}
R_{ijk\ell}(x) &= (F^{L,*}R)_{ijk\ell}(x) = R_{\lambda\mu\nu\kappa}(F(x^0,\theta,x))
\frac{\partial F^{L, \lambda}}{\partial x^i}\frac{\partial F^{L, \mu}}{\partial x^j}
\frac{\partial F^{L, \nu}}{\partial x^k}\frac{\partial F^{L, \kappa}}{\partial x^{\ell}}\nonumber\\
K_{i_0 i_1\ldots i_{2k-1}}(x) &=(F^{L,*}K)_{i_0 i_1\ldots i_{2k-1}}(x)
 = K_{\lambda_0\lambda_1\ldots\lambda_{2k-1}}(F(x^0,\theta,x)) \lambda_1\ldots\lambda_{2k-1}\\
&\qquad \cdot \frac{\partial F^{L, \lambda_0}}{\partial x^{i_0}}\frac{\partial F^{L, \lambda_1}}{\partial x^{i_1}}\cdot\ldots\cdot
\frac{\partial F^{L, \lambda_{2k-1}}}{\partial x^{i_{2k-1}}}.
\nonumber
\end{align}

The following computation finishes the proof of Prop.~\ref{prop:2.1}.
\begin{lem}\label{lem:2.3}
    Under the assumption (\ref{eq:isom}), we have
    $$F^{L,*} d_{L{\bar M}} CS^W=0.$$ 
\end{lem}

\begin{proof}
   By Lemma~\ref{lem:2.2}, at a fixed $x^0$, we have
    \begin{align*}
 \MoveEqLeft{   F^{L,*} d_{L{\bar M}} CS^W    (\partial_{x^0}, \partial_{x^1}, \cdots, \partial_{x^{2k-1}})|_{x} }\\
 &= 
  \int_0^{2\pi} K_{\lambda_0 \lambda_1\ldots\lambda_{2k-1}}(F(x^0,\theta,x))
\frac{\partial \alpha^i}{\partial\theta} \frac{\partial F^{\lambda_0}}{\partial x^i}
\frac{\partial F^{\lambda_1}}{\partial x^1}
\cdots
\frac{\partial F^{\lambda_{2k-1}}}{\partial x^{2k-1}}
 d\theta\\
 &=\int_0^{2\pi}\frac{\partial \alpha^i}{\partial\theta} K_{i1\ldots 2k-1}(x) d\theta
 = K_{i1\ldots 2k-1}(x)\int_0^{2\pi}\frac{\partial \alpha^i}{\partial\theta} d\theta\\
 &=0
    \end{align*}
    using (\ref{eq:11}).  As in (\ref{CSW}), the integration over $[0, 2\pi]$ is valid, because the $\alpha^i$ are the components of a tensor/vector (\ref{eq:five}). 
\end{proof}

\section{Other corrections}
We list other corrections by section.

\subsection{Corrections in \S3}
\begin{itemize}
\item The equation  $\alpha \mapsto \alpha' \mapsto \alpha'_*[M]$ on p.~497 should be deleted. 
    \item In Prop. 3.1, $\eta$ is not a homomorphism, but this claimed result is not used anywhere in the paper.
    \item The right hand side of (3.5) is missing $\int_{S^1}.$
\end{itemize}

\subsection{Corrections in Appendix B} 

Throughout the paper, we used a high Sobolev topology or the Fr\'echet topology on $LM$, and we claimed that the homotopy type of $LM$ is independent of this choice. We claimed the same for $\diff(M)$.  However, we can only prove this results for the $C^k$ topologies on $LM$, $k\in [0, \infty)$, as we now show.  As a result, we should use $C^k$ topologies in the paper.  However, while $\diff(M)$ is infinite dimensional,  $\ism(M)$ is a finite dimensional Lie group
\cite[vol.~I, p.~239]{K-N}, and every differentiable isometry is smooth\footnote{
{\tt https://mathoverflow.net/questions/262058/a-differentiable-isometry-is-smooth} }.  Therefore, we don't need the result below, but we include it to correct the wrong proof in \cite{MRT4}.

We use that the space $C^k(N,M)$ of smooth maps between closed manifolds $N,M$ is a smooth Banach manifold for $k$ as above \cite[\S11]{A}.  We also need the following:

\begin{thm} \cite[Thm.~16]{palais} Let $V_1, V_2$ be locally convex topological vector spaces 
({\it e.g.,} Banach spaces) with $f:V_1\to V_2$ a continuous linear map with dense image.  Let $O$ be open in $V_2$ and set $\tilde O = f^{-1}(O).$  If $V_1, V_2$ are metrizable (or more generally that $O,\tilde O$ are paracompact), then
$f|_{\tilde O}:\tilde O\to O$ is a homotopy equivalence.
\end{thm}

We now prove a corrected version of Lem.~B.1.

\begin{lem}\label{lem:3.1}
    For  $k,k'\in [0,\infty)$ and $N,M$ closed manifolds,  $C^k(N,M)$ and $C^{k'}(N,M)$ 
are homotopy equivalent smooth Banach manifolds.
\end{lem}

It is plausible that ${\rm Diff}^k(M)$ and ${\rm Diff}^{k'}(M)$  are also homotopy equivalent,
but we don't have a proof for this.

\begin{proof}
Let $M\to \R^\ell$ be an embedding of $M$ into Euclidean space.  Set $V_1 = C^{k'}(N,\R^\ell), V_2 = C^k(N,\R^\ell)$ with their Banach space norm, for $k'>k.$ 
Let $f=i:V_1\to V_2$ be the continuous inclusion.  Let $T$ be an open tubular neighborhood of $M$ in $\R^\ell$.  Set $O = C^k(N,T)$, so $\tilde O = C^{k'}(N,T).$  $O$ is open
in $V_2$, so $C^k(N,T)$ and $ C^{k'}(N,T)$ are homotopy equivalent.  Since $T$ deformation retracts onto $M$, $C^k(M,T)$ deformation retracts onto $C^k(N,M)$, and similarly for $C^{k'}.$  Thus $C^k(N,T)$ is homotopy equivalent to  $C^k(N,M)$.  We conclude that $C^k(N,M)$ and $C^{k'}(N,M)$ are homotopy equivalent.
\end{proof}

\appendix
\section{Proofs of Lemma \ref{lem:2.2a} and Lemma \ref{lem:2.3a}}

\noindent {\it Proof of Lemma \ref{lem:2.2a}.}
 We do the case ${\dim}(M) = 2k-1=5$   to keep the notation down.  Fix  
 $x\in \bm$ and $\xi\in T_x\bm.$ 
 For $X_0, X_1,\ldots, X_5\in T_x\bm,$ set
 \begin{equation}\label{eq:tk}\tilde K(X_0,\ldots, X_5)_x = \sum_{a=0}^5 \partial_{\lambda^a} K_{\nu \lambda_0 \cdots \widehat{\lambda_a}\cdots\lambda_{2k-1}} 
(x)\xi^\nu
X_a^{\lambda_a} X_0^{\lambda_0}\ldots \widehat{X_a^{\lambda_a}} \ldots X_5^{\lambda_5}.
\end{equation}
If we show that the right hand side of (\ref{eq:tk}) is skew-symmetric in $X_0,\ldots, X_5$, then
$\tilde K(X_0,\ldots, X_5)$ is a $6$-form on $\bm$ and hence must vanish.  The lemma follows by replacing $x$ with
$F(x^0,\theta,x)$, $\xi$ with $(d/d\theta)F(x^0,\theta, x)$, and $X_i^{\lambda_i}$ with $\partial F^{\lambda_i}/\partial x^i.$

To check skew-symmetry in $X_0, X_1$, we write
\begin{align}
\MoveEqLeft{\tilde K(X_0,X_1, X_2, X_3, X_4, X_5) }\nonumber\\
&= (\partial_{\lambda_0} K_{\nu \lambda_1\lambda_2\lambda_3\lambda_4\lambda_5}
- \partial_{\lambda_1} K_{\nu \lambda_0 \lambda_2\lambda_3\lambda_4\lambda_5}) \xi^\nu
X_0^{\lambda_0} X_1^{\lambda_1} X_2^{\lambda_2}X_3^{\lambda_3}X_4^{\lambda_4} X_5^{\lambda_5} \label{16}\\
&\qquad + (\partial_{\lambda_2} K_{\nu \lambda_0 \lambda_1 \lambda_3\lambda_4\lambda_5}
-\partial_{\lambda_3} K_{\nu \lambda_0 \lambda_1\lambda_2\lambda_4\lambda_5}
+ \partial_{\lambda_4} K_{\nu \lambda_0 \lambda_1\lambda_2\lambda_3\lambda_5}
-\partial_{\lambda_5} K_{\nu \lambda_0 \lambda_1\lambda_2\lambda_3\lambda_4})\label{17}\\
&\qquad\qquad \cdot \xi^\nu X_0^{\lambda_0} X_1^{\lambda_1} X_2^{\lambda_2}X_3^{\lambda_3}X_4^{\lambda_4} X_5^{\lambda_5},\nonumber\\
\MoveEqLeft{\tilde K(X_1,X_0, X_2, X_3, X_4, X_5) }\nonumber\\
&= (\partial_{\lambda_1} K_{\nu \lambda_0\lambda_2\lambda_3\lambda_4\lambda_5}
- \partial_{\lambda_0} K_{\nu \lambda_1 \lambda_2\lambda_2\lambda_4\lambda_5}) \xi^\nu
X_1^{\lambda_1} X_0^{\lambda_0} X_2^{\lambda_2}X_3^{\lambda_3}X_4^{\lambda_4} X_5^{\lambda_5}  \label{18}\\
&\qquad + (\partial_{\lambda_2} K_{\nu \lambda_1 \lambda_0 \lambda_3\lambda_4\lambda_5}
-\partial_{\lambda_3} K_{\nu \lambda_1 \lambda_0\lambda_2\lambda_4\lambda_5}
+ \partial_{\lambda_4} K_{\nu \lambda_1 \lambda_0\lambda_2\lambda_3\lambda_5}
-\partial_{\lambda_5} K_{\nu \lambda_1 \lambda_0\lambda_2\lambda_3\lambda_4})   \label{19}\\
&\qquad\qquad \cdot \xi^\nu X_1^{\lambda_1} X_0^{\lambda_0} X_2^{\lambda_2}X_3^{\lambda_3}X_4^{\lambda_4} X_5^{\lambda_5}.\nonumber
\end{align}
Then $\text{(\ref{16})} = - \text{(\ref{18})}$ by inspection, and $\text{(\ref{17})} = - \text{(\ref{19})}$, because  $K$ is skew-symmetric in $\lambda_1,\ldots,\lambda_5$ by (\ref{K}).

We now check skew-symmetry in $X_1, X_2$, with all other cases being similar.  We have
\begin{align}
\MoveEqLeft{\tilde K(X_0,X_2, X_1, X_3, X_4, X_5) }\nonumber\\
&= \partial_{\lambda_0} K_{\nu \lambda_2\lambda_1\lambda_3\lambda_4\lambda_5}
\xi^\nu X_1^{\lambda_1} X_0^{\lambda_0} X_2^{\lambda_2}X_3^{\lambda_3}X_4^{\lambda_4} X_5^{\lambda_5}\label{20}\\
&\qquad +( - \partial_{\lambda_2} K_{\nu \lambda_0\lambda_1\lambda_3\lambda_4\lambda_5} 
+ \partial_{\lambda_1} K_{\nu \lambda_0\lambda_2\lambda_3\lambda_4\lambda_5} )
\xi^\nu X_1^{\lambda_1} X_0^{\lambda_0} X_2^{\lambda_2}X_3^{\lambda_3}X_4^{\lambda_4} X_5^{\lambda_5}\label{21}\\
&\qquad +(- \partial_{\lambda_3} K_{\nu \lambda_0\lambda_2\lambda_1\lambda_4\lambda_5} 
+ \partial_{\lambda_4} K_{\nu \lambda_0\lambda_2\lambda_1\lambda_3\lambda_5} 
+ \partial_{\lambda_5} K_{\nu \lambda_0\lambda_2\lambda_1\lambda_3\lambda_4}) \label{22}\\
&\qquad \qquad \cdot\xi^\nu  X_0^{\lambda_0} X_2^{\lambda_2}X_1^{\lambda_1}X_3^{\lambda_3}X_4^{\lambda_4} X_5^{\lambda_5}.
\nonumber
\end{align}  
Then $\tilde K(X_0,X_2, X_1, X_3, X_4, X_5) = - \tilde K(X_0,X_1, X_2, X_3, X_4, X_5) $, because 
(i) the skew-symmetry of $K$ implies the skew-symmetry of (\ref{20}) and (\ref{22}) in $\lambda_1,\lambda_2$; (ii)
(\ref{21}) is explicitly skew-symmetric in $\lambda_1,\lambda_2$.

The general case is similar.
\hfill $\Box$
\bigskip

\noindent {\it Proof of Lemma \ref{lem:2.3a}.}  We again do the case ${\rm dim}(M) = 5,$ with the general case being similar.  The terms with second partial derivatives are
\begin{align}
&K_{\nu\lambda_1\lambda_2\lambda_3\lambda_4\lambda_5}\left(\alpha^1 \frac{\partial^2F^\nu}{\partial x^1\partial\theta}+ \alpha^2 \frac{\partial^2F^\nu}{\partial x^2\partial\theta}
+\alpha^3 \frac{\partial^2F^\nu}{\partial x^3\partial\theta}
+ \alpha^4 \frac{\partial^2F^\nu}{\partial x^4\partial\theta}
+\alpha^5 \frac{\partial^2F^\nu}{\partial x^5\partial\theta}
\right) \label{23}\\
&\qquad \qquad \cdot \frac{\partial F^{\lambda_1}}{\partial x^1}\frac{\partial F^{\lambda_2}}{\partial x^2}
\frac{\partial F^{\lambda_3}}{\partial x^3}\frac{\partial F^{\lambda_4}}{\partial x^4}
\frac{\partial F^{\lambda_5}}{\partial x^5}\nonumber\\
&-K_{\nu\lambda_0\lambda_2\lambda_3\lambda_4\lambda_5}\frac{\partial^2F^\nu}{\partial x^1\partial\theta}
\left( \alpha^i\frac{\partial F^{\lambda_0}}{\partial x^i}\right)
\frac{\partial F^{\lambda_2}}{\partial x^2}
\frac{\partial F^{\lambda_3}}{\partial x^3}\frac{\partial F^{\lambda_4}}{\partial x^4}
\frac{\partial F^{\lambda_5}}{\partial x^5}
\label{24}\\
&+ K_{\nu\lambda_0\lambda_1\lambda_3\lambda_4\lambda_5}\frac{\partial^2F^\nu}{\partial x^2\partial\theta}
\left( \alpha^i\frac{\partial F^{\lambda_0}}{\partial x^i}\right)
\frac{\partial F^{\lambda_1}}{\partial x^1}\frac{\partial F^{\lambda_3}}{\partial x^3}
\frac{\partial F^{\lambda_4}}{\partial x^4}\frac{\partial F^{\lambda_5}}{\partial x^5}
\label{25}\\
&- K_{\nu\lambda_0\lambda_1\lambda_2\lambda_4\lambda_5}\frac{\partial^2F^\nu}{\partial x^3\partial\theta}
\left( \alpha^i\frac{\partial F^{\lambda_0}}{\partial x^i}\right)
\frac{\partial F^{\lambda_1}}{\partial x^1}\frac{\partial F^{\lambda_2}}{\partial x^2}
\frac{\partial F^{\lambda_4}}{\partial x^4}\frac{\partial F^{\lambda_5}}{\partial x^5}
\label{26}\\
&+ K_{\nu\lambda_0\lambda_1\lambda_2\lambda_3\lambda_5}\frac{\partial^2F^\nu}{\partial x^4\partial\theta}
\left( \alpha^i\frac{\partial F^{\lambda_0}}{\partial x^i}\right)
\frac{\partial F^{\lambda_1}}{\partial x^1}\frac{\partial F^{\lambda_2}}{\partial x^2}
\frac{\partial F^{\lambda_3}}{\partial x^3}\frac{\partial F^{\lambda_5}}{\partial x^5}
\label{27}\\
&- K_{\nu\lambda_0\lambda_1\lambda_2\lambda_3\lambda_4}\frac{\partial^2F^\nu}{\partial x^5\partial\theta}
\left( \alpha^i\frac{\partial F^{\lambda_0}}{\partial x^i}\right)
\frac{\partial F^{\lambda_1}}{\partial x^1}\frac{\partial F^{\lambda_2}}{\partial x^2}
\frac{\partial F^{\lambda_3}}{\partial x^3}\frac{\partial F^{\lambda_4}}{\partial x^4}.
\label{28}
\end{align}
In (\ref{24}), in the term $\alpha^i (\partial F^{\lambda_0}/\partial x^i)$, only the term $\alpha^1 (\partial F^{\lambda_0}/\partial x^1)$ is nonzero: for example, the term 
$$K_{\nu\lambda_0\lambda_2\lambda_3\lambda_4\lambda_5}\frac{\partial^2F^\nu}{\partial x^1\partial\theta}
\left( \alpha^2\frac{\partial F^{\lambda_0}}{\partial x^2}\right) \frac{\partial F^{\lambda_2}}{\partial x^2}
\frac{\partial F^{\lambda_3}}{\partial x^3}\frac{\partial F^{\lambda_4}}{\partial x^4}
\frac{\partial F^{\lambda_5}}{\partial x^5}$$
is skew-symmetric in $\lambda_0, \lambda_2$, and so vanishes. For the same reasons, the terms with 
$\partial F^{\lambda_0}/\partial x^3,$   $ \partial F^{\lambda_0}/\partial x^4,$   $ \partial F^{\lambda_0}/\partial x^5$ vanish.
  Similarly, in (\ref{25}) only $\alpha^2 (\partial F^{\lambda_0}/\partial x^2)$ is nonzero, in (\ref{26}) only $\alpha^3 (\partial F^{\lambda_0}/\partial x^3)$ is nonzero, 
in (\ref{27}) only $\alpha^4 (\partial F^{\lambda_0}/\partial x^4)$ is nonzero, and
in (\ref{28}) only $\alpha^5 (\partial F^{\lambda_0}/\partial x^5)$ is nonzero.

Thus (\ref{23}) -- (\ref{28}) becomes
\begin{align}
&K_{\nu\lambda_1\lambda_2\lambda_3\lambda_4\lambda_5}\left(\alpha^1 \frac{\partial^2F^\nu}{\partial x^1\partial\theta}+ \alpha^2 \frac{\partial^2F^\nu}{\partial x^2\partial\theta}
+\alpha^3 \frac{\partial^2F^\nu}{\partial x^3\partial\theta}
+ \alpha^4 \frac{\partial^2F^\nu}{\partial x^4\partial\theta}
+\alpha^5 \frac{\partial^2F^\nu}{\partial x^5\partial\theta}
\right) \label{29}\\
&\qquad \qquad \cdot \frac{\partial F^{\lambda_1}}{\partial x^1}\frac{\partial F^{\lambda_2}}{\partial x^2}
\frac{\partial F^{\lambda_3}}{\partial x^3}\frac{\partial F^{\lambda_4}}{\partial x^4}
\frac{\partial F^{\lambda_5}}{\partial x^5}\nonumber\\
&-K_{\nu\lambda_0\lambda_2\lambda_3\lambda_4\lambda_5}\frac{\partial^2F^\nu}{\partial x^1\partial\theta}
\left( \alpha^1\frac{\partial F^{\lambda_0}}{\partial x^1}\right)
\frac{\partial F^{\lambda_2}}{\partial x^2}
\frac{\partial F^{\lambda_3}}{\partial x^3}\frac{\partial F^{\lambda_4}}{\partial x^4}
\frac{\partial F^{\lambda_5}}{\partial x^5}
\label{30}\\
&+ K_{\nu\lambda_0\lambda_1\lambda_3\lambda_4\lambda_5}\frac{\partial^2F^\nu}{\partial x^2\partial\theta}
\left( \alpha^2\frac{\partial F^{\lambda_0}}{\partial x^2}\right)
\frac{\partial F^{\lambda_1}}{\partial x^1}\frac{\partial F^{\lambda_3}}{\partial x^3}
\frac{\partial F^{\lambda_4}}{\partial x^4}\frac{\partial F^{\lambda_5}}{\partial x^5}
\label{31}\\
&- K_{\nu\lambda_0\lambda_1\lambda_2\lambda_4\lambda_5}\frac{\partial^2F^\nu}{\partial x^3\partial\theta}
\left( \alpha^3\frac{\partial F^{\lambda_0}}{\partial x^3}\right)
\frac{\partial F^{\lambda_1}}{\partial x^1}\frac{\partial F^{\lambda_2}}{\partial x^2}
\frac{\partial F^{\lambda_4}}{\partial x^4}\frac{\partial F^{\lambda_5}}{\partial x^5}
\label{32}\\
&+ K_{\nu\lambda_0\lambda_1\lambda_2\lambda_3\lambda_5}\frac{\partial^2F^\nu}{\partial x^4\partial\theta}
\left( \alpha^4\frac{\partial F^{\lambda_0}}{\partial x^4}\right)
\frac{\partial F^{\lambda_1}}{\partial x^1}\frac{\partial F^{\lambda_2}}{\partial x^2}
\frac{\partial F^{\lambda_3}}{\partial x^3}\frac{\partial F^{\lambda_5}}{\partial x^5}
\label{33}\\
&- K_{\nu\lambda_0\lambda_1\lambda_2\lambda_3\lambda_4}\frac{\partial^2F^\nu}{\partial x^5\partial\theta}
\left( \alpha^5\frac{\partial F^{\lambda_0}}{\partial x^5}\right)
\frac{\partial F^{\lambda_1}}{\partial x^1}\frac{\partial F^{\lambda_2}}{\partial x^2}
\frac{\partial F^{\lambda_3}}{\partial x^3}\frac{\partial F^{\lambda_4}}{\partial x^4}.
\label{34}
\end{align}

If we replace $\lambda_0$ in (\ref{30}) with $\lambda_1$, then the term 
$$K_{\nu\lambda_1\lambda_2\lambda_3\lambda_4\lambda_5}\left(\alpha^1 \frac{\partial^2F^\nu}{\partial x^1\partial\theta}\right)\frac{\partial F^{\lambda_1}}{\partial x^1}\frac{\partial F^{\lambda_2}}{\partial x^2}
\frac{\partial F^{\lambda_3}}{\partial x^3}\frac{\partial F^{\lambda_4}}{\partial x^4}
\frac{\partial F^{\lambda_5}}{\partial x^5}$$
in (\ref{29}) cancels with (\ref{30}).
If we replace $\lambda_0$ in (\ref{31}) with $\lambda_2$, then the term 
$$K_{\nu\lambda_1\lambda_2\lambda_3\lambda_4\lambda_5}\left(\alpha^2 \frac{\partial^2F^\nu}{\partial x^2\partial\theta}\right)\frac{\partial F^{\lambda_1}}{\partial x^1}\frac{\partial F^{\lambda_2}}{\partial x^2}
\frac{\partial F^{\lambda_3}}{\partial x^3}\frac{\partial F^{\lambda_4}}{\partial x^4}
\frac{\partial F^{\lambda_5}}{\partial x^5}$$
in (\ref{29}) cancels with (\ref{31}).
If we replace $\lambda_0$ in (\ref{32}) with $\lambda_3$, then the term 
$$K_{\nu\lambda_1\lambda_2\lambda_3\lambda_4\lambda_5}\left(\alpha^3 \frac{\partial^2F^\nu}{\partial x^3\partial\theta}\right)\frac{\partial F^{\lambda_1}}{\partial x^1}\frac{\partial F^{\lambda_2}}{\partial x^2}
\frac{\partial F^{\lambda_3}}{\partial x^3}\frac{\partial F^{\lambda_4}}{\partial x^4}
\frac{\partial F^{\lambda_5}}{\partial x^5}$$
in (\ref{29}) cancels with (\ref{32}).
If we replace $\lambda_0$ in (\ref{33}) with $\lambda_4$, then the term 
$$K_{\nu\lambda_1\lambda_2\lambda_3\lambda_4\lambda_5}\left(\alpha^4 \frac{\partial^2F^\nu}{\partial x^4\partial\theta}\right)\frac{\partial F^{\lambda_1}}{\partial x^1}\frac{\partial F^{\lambda_2}}{\partial x^2}
\frac{\partial F^{\lambda_3}}{\partial x^3}\frac{\partial F^{\lambda_4}}{\partial x^4}
\frac{\partial F^{\lambda_5}}{\partial x^5}$$
in (\ref{29}) cancels with (\ref{33}).
If we replace $\lambda_0$ in (\ref{34}) with $\lambda_5$, then the term 
$$K_{\nu\lambda_1\lambda_2\lambda_3\lambda_4\lambda_5}\left(\alpha^5 \frac{\partial^2F^\nu}{\partial x^5\partial\theta}\right)\frac{\partial F^{\lambda_1}}{\partial x^1}\frac{\partial F^{\lambda_2}}{\partial x^2}
\frac{\partial F^{\lambda_3}}{\partial x^3}\frac{\partial F^{\lambda_4}}{\partial x^4}
\frac{\partial F^{\lambda_5}}{\partial x^5}$$
in (\ref{29}) cancels with (\ref{34}).

Thus (29) -- (34) sum to zero, which proves the Lemma.  \hfill $\Box$

 \bibliographystyle{amsplain}
\bibliography{GeometryofLoopSpacesCorrections}

\providecommand{\bysame}{\leavevmode\hbox to3em{\hrulefill}\thinspace}
\providecommand{\MR}{\relax\ifhmode\unskip\space\fi MR }
\providecommand{\MRhref}[2]{%
  \href{http://www.ams.org/mathscinet-getitem?mr=#1}{#2}
}
\providecommand{\href}[2]{#2}
\begin{thebibliography}{10}

\bibitem{A}
{Abraham, R}, \emph{Lectures of {S}male on differential topology},
  \url{https://www.math.cuhk.edu.hk/course_builder/1415/math5070/Smale--Lectures%20on%20Differential%20Topology%201962.pdf}.

\bibitem{BK}
{Bamler, R.} and {Kleiner, B.}, \emph{Ricci flow and diffeomorphism groups of
  3-manifolds}, J. Amer. Math. Soc. \textbf{36} (2023), no.~2, 563--589.

\bibitem{Bar}
{Baraglia, D.}, \emph{On the mapping class groups of simply-connected smooth
  4-manifolds}, \url{https://arxiv.org/abs/2310.18819}.

\bibitem{IKMT}
{Iida, N.}, {Konno, H.}, {Mukherjee, A.}, and {Taniguchi, M.},
  \emph{Diffeomorphisms of 4-manifolds with boundary and exotic embeddings},
  \url{https://arxiv.org/abs/2203.14878}.

\bibitem{K-N}
{Kobayashi, S.} and {Nomizu, K.}, \emph{Foundations of {D}ifferential
  {G}eometry}, vol.~2, Interscience Publishers, John Wyley \& Sons Inc., New
  York, 1969.

\bibitem{KT}
{Konno, H.} and {Taniguchi, M.}, \emph{The groups of diffeomorphisms and
  homeomorphisms of 4-manifolds with boundary}, Adv. Math. \textbf{409} (2022).

\bibitem{KM}
{Kriegl, A.} and {Michor, P.}, \emph{The {C}onvenient {S}etting of {G}lobal
  {A}nalysis}, American Mathematical Society, Providence, RI, 1997.

\bibitem{KrM}
{Kronheimer, P.} and {Mrowka, T.}, \emph{The {D}ehn twist on a sum of two {K}3
  surfaces}, Math. Res. Lett. \textbf{27(6)} (2020), 1767--1783.

\bibitem{KRW}
{Kupers, A.} and {Randal-Williams, O.}, \emph{On diffeomorphisms of
  even-dimensional discs}, \url{https://arxiv.org/abs/2007.13884}.

\bibitem{LM}
{Lin, J.} and {Mukherjee, A.}, \emph{Family {B}auer--{F}uruta invariant, exotic
  surfaces and {S}male conjecture}, \url{https://arxiv.org/abs/2110.09686},
  2021.

\bibitem{MRT4}
{Maeda, Y.}, {Rosenberg, S.}, and {Torres-Ardila, F.}, \emph{The geometry of
  loop spaces {{\rm I}}{{\rm I}}: Characteristic classes}, Adv. in Math.
  \textbf{287} (2016), 485--518.

\bibitem{palais}
{Palais, R.}, \emph{Homotopy theory of infinite dimensional manifolds},
  Topology \textbf{5} (1966), 1--15.

\bibitem{W}
{Watanabe, T.}, \emph{Some exotic nontrivial elements of {D}iff(${S}^4$)},
  \url{https://arxiv.org/abs/1812.02248}.

\end{thebibliography}

\end{document}